\documentclass[twoside,a4paper,12pt]{amsart}
\usepackage{amsfonts, amsbsy, amsmath, amsthm, amssymb, latexsym}
\usepackage{mathrsfs, mathabx}
\usepackage[top=34mm,right=20mm,bottom=34mm,left=20mm]{geometry}
\usepackage{bm}
\usepackage{fancyhdr}
\usepackage{enumerate}
\usepackage{tikz, tikz-cd, tkz-euclide}
\usetikzlibrary{decorations.markings,positioning}
\usepackage{xcolor}
\usepackage{faktor}

\usepackage{hyperref}

\usepackage[official]{eurosym}
\usepackage{array,float,graphicx,epic,eepic}
\usepackage{amssymb, amsmath, amsthm}
\usepackage{changepage}
\usepackage{comment}
\usepackage{thmtools,thm-restate}

\usepackage{url}

\def \G {\Gamma}

\def \Aut {\text{Aut}}

\renewcommand{\d}{\delta}

\newcommand{\R}{\mathbb{R}}
\newcommand{\Z}{\mathbb{Z}}

 \newcommand{\e}{\varepsilon}

 \renewcommand{\to}{\rightarrow}

\newcommand{\Q}{\mathbb{Q}}

\newtheorem{thm}{Theorem}[section]
\newtheorem{prop}[thm]{Proposition}
\newtheorem{lem}[thm]{Lemma}
\newtheorem{cor}[thm]{Corollary}

\theoremstyle{definition}
\newtheorem{defn}[thm]{Definition}
\newtheorem*{rem}{Remark}
\newtheorem*{quest}{Questions}

\newtheorem{exs}[thm]{Examples}

\title{Length functions on groups and actions on graphs}

\author{M. Collins and A. Martino}

\address{Mathematical Sciences, Building 54, University of Southampton, Southampton, SO17 1BJ}
\email{mpc1g15@soton.ac.uk}
\email{A.Martino@soton.ac.uk}

\date{}

\begin{document}

\begin{abstract}
	We study generalisations of Chiswell's Theorem that $0$-hyperbolic Lyndon length functions on groups always arise as based length functions of the the group acting isometrically on a tree. We produce counter-examples to show that this Theorem fails if one replaces $0$-hyperbolicity with $\delta$-hyperbolicity. 
	
	We then propose a set of axioms for the length function on a finitely generated group that ensures the function is bi-Lipschitz equivalent to a (or any) length function of the group acting on its Cayley graph. 
\end{abstract}

	\maketitle
	
	\section{Introduction}
	
	One of the key insights of geometric group theory is that one can obtain information on a group by viewing it as a metric space, via the word metric on its Cayley graph. More generally if a group, $G$, acts isometrically on a metric space $(X,d)$ one can elucidate properties of the group from this action. For instance, the class of hyperbolic groups is precisely the class of those groups admitting a proper, co-compact isometric action on some locally compact, geodesic $\delta$-hyperbolic space $X$. 
	
	Given a (right) isometric action of $G$ on $(X,d)$, and a point $p$ in $X$, one can define a $G$-invariant pseudo-metric - which we denote by $d_p$ - on $G$ via $d_p(g,h) := d(pg, ph)$, which is a metric precisely when the stabiliser of $p$ is trivial. In fact, this metric on $G$ can be encoded via the \textit{based length function}. 
	
	\begin{defn}
		\label{basedlength}
		Let $G$ act isometrically on the metric space $(X,d)$. Then the \textit{based length function} of $G$ based at some point, $p \in X$ is the function, $l_p:G \to \R$, given by: 
		$$
			l_p(g)  :=  d(p, pg) 
		$$
	\end{defn}
	
	It is straightforward to see that one can recover the invariant (pseudo) metric from the based length function via $d_p(g,h) = l_p(gh^{-1})$. 
	
	Of course, in order to obtain properties of the group it is helpful to impose conditions on the space and the action, just as for hyperbolicity above. A key area where one can recover a great deal of information about $G$ is when $X$ is a tree.

	The source of inspiration for this paper is a striking result of Chiswell, that one can axiomatise the based length functions arising from actions on trees  - sometimes called Lyndon length functions, following results from \cite{Lyndon1963} - and, from the axioms, always recover an isometric action. Specifically, 
	
	\begin{thm}[\cite{Chiswell1976}]
		\label{chiswell}
		Suppose $l:G \to \R_{\geq 0}$ satisfies the following axioms: 
			\begin{enumerate}[{A}1:]
			\item[{A}1':] $l(g)=0$ if  $g=1$
			\item[A2\phantom{'}:] $l(g^{-1}) = l(g)$
			\item[A3\phantom{'}:] $c(g,h) \geqq 0$
			\item[$H_0$\phantom{'}:] For all $g_1, g_2, g_3 \in G$, \\ $c(g_1,g_2)\geq m$, $c(g_2,g_3)\geq m$ implies that $c(g_1, g_3)\geq m$ \hspace{35pt},
		\end{enumerate}
		where 
		$$
		c(g,h) := \frac{1}{2}(l(g) + l(h) - l(gh^{-1})).
		$$
		Then there exists an $\R$-tree, $(X,d)$, admitting an isometric $G$-action and a point, $p \in X$, such that $l_p(g) = l(g)$. Moreover, if the images of $l$ and $c$ lie in $\Z$, then the tree will be simplicial.  
	\end{thm}
	
\begin{rem}
		As noted above a function $d:G \times G \to \R$  can be defined from $l$ and, from this point of view, A1' says that $d$ vanishes on the diagonal, A2 says that it is symmetric and A3 says that it satisfies the triangle inequality. 
		
		The function $c(g,h)$ is then really the Gromov product and axiom $H_0$ should be thought of as a $0$-hyperbolicity condition (see, for example, \cite{Alonso1991} for a discussion on hyperbolic groups, spaces and the Gromov product). Chiswell's Theorem can then be summarised as saying that a $0$-hyperbolic Lyndon length function is always a based length function on a $0$-hyperbolic space.

\end{rem}	

%
%

With this in mind, we are motivated to ask the following questions. 

\begin{quest}
	
	\medspace
	
	\begin{itemize}
		\item Is there a generalisation of Chiswell's Theorem for isometric group actions on metric graphs?
		\item In particular, is there a generalisation of Chiswell's Theorem for isometric actions on $\delta$-hyperbolic graphs? 
	\end{itemize}
\end{quest}	

\begin{rem}
	In the spirit of Chiswell's result, we will consider graphs whose edge lengths may not be integers. For instance, one could take the Cayley graph of a group with respect to some generating set, and then equivariantly assign positive real lengths to edges.  
\end{rem}

It turns out that these questions are somehow too broad in their scope. Given a (strictly positive) length function on $G$ (see Definition~\ref{hyplength} for the definition of a length function) there is always a metric graph whose based length function is equal to this function: take the complete graph on $G$ where the edge between $g$ and $h$ has length $l(hg^{-1})$ - Lemma~\ref{complete}. The based length function on this graph, with respect to the basepoint $1$, is equal to $l$. However, this action is not particularly useful. 

In order to rule out this kind of example we will add some restrictions. 

\begin{quest}
	Let us suppose that $G$ is finitely generated and let us restrict ourselves to isometric, co-compact actions on locally compact graphs, $X$.  
	
	\medspace
	
	\begin{itemize}
		\item Given a (strictly positive) length function, $l$, does $G$ admit an isometric, co-compact action on a locally compact metric graph, $X$, such that $l=l_p$ for some $p \in X$?
		\item What if we add the hypothesis that $l$ is $\delta$-hyperbolic (see Definition~\ref{hyplength} for the definition of hyperbolicity)? 
	\end{itemize}
\end{quest}

It turns out that the answer to both of these questions is \textbf{no}. By Proposition~\ref{deltacounterexample}, there exists a $\delta$-hyperbolic length function which cannot arise as the based length function associated to any isometric, co-compact action on a locally compact graph. 

However, that example is bi-Lipschitz equivalent to a length function on a Cayley graph. (Note that, for a finitely generated groups, all based length functions on Cayley graphs with respect to finite generated sets are bi-Lipschitz equivalent). But one can also produce examples of $\delta$-hyperbolic length functions which are not bi-Lipschitz equivalent to any based length function on a Cayley graph, as in Proposition~\ref{logs}. In fact, \emph{every} finitely generated group admits a hyperbolic length function.

	\begin{restatable*}{thm}{badChiswell}
		\label{badChiswell}
	There exists a finitely generated group, $G$, with a hyperbolic length function, 	$l:G \to \R_{\geq 0}$ such that $l \neq l_p$ for any co-compact, metric $G$-graph.

	Moreover, \emph{any} finitely generated group admits a (free) hyperbolic length function. In particular, we can find an example of a group $G$ with a hyperbolic length function, $l$, which is not quasi-isometric to any based length function arising from an isometric action of $G$ on a geodesic and proper $\delta$-hyperbolic metric space.  
\end{restatable*}

%

This leads us to the following. 

\begin{quest}
	
	Suppose that $G$ is finitely generated. 
	
	
	\begin{itemize}
		\item Can one axiomatise those length functions which are bi-Lipschitz equivalent to some (and hence all) based length functions on a Cayley graph for $G$ (with respect to a finite generating set)? 
		\item Can we make these axioms apply to - for instance - any free $F_n$ action on a simplicial tree as well as Cayley graphs?
		\item Does this axiomatisation define a connected/contractible/finite dimensional subspace of $\R^G$ on which $\Aut(G)$ acts?
	\end{itemize}
	
\end{quest}

\begin{rem}
	We do come up with a axiom scheme, below, and we observe that these axioms hold for all sufficiently well behaved actions - see Proposition~\ref{generalspace} and Corollary~\ref{preA5} - and in particular to all points of Culler-Vogtmann space.
	
	The third question here arises from the fact that one key use of Chiswell's Theorem is in the study of group actions on trees, and the definition of the space of such actions which are then encoded via functions (usually the translation length function, which is related to the Lyndon length function). See \cite{Culler1986} for the seminal paper on the `Outer Space' of free actions on trees, encoded by length functions (amongst other things). 
	
	It is clear that the space of all length functions which are bi-Lipscitiz to one arising from a Cayley graph is a contractible space (because a linear combination of such functions is another such function). Therefore, this provides a contractible space on which $\Aut(G)$ acts. However, it is far too large and so one might hope that an axiomatisation could provide a more reasonable subspace. 
\end{rem}


With these questions in mind, we propose the following axioms for our length functions:

\begin{restatable*}{defn}{axioms}
	\label{axioms}
	Let $G$ be a group. We say that $l: G \to \R_{\geq 0}$ is a \emph{graph-like length function} if it satisfies the following axioms:
	\begin{enumerate}[{A}1:]
		\item $l(g)=0$ if and only if $g = 1$
		\item $l(g^{-1}) = l(g)$
		\item $c(g,h) \geqq 0$
		\item For all $R\geq 0$, the closed ball $B_R := \{g\in G \mid l(g) \leq R\}$ is finite
		\item There exists $0 \leq \epsilon < 1$ and $K>0$ such that, for any $g \in G$, if $l(g) > K$ then there exists an $x \in G$ with: 
		\begin{enumerate}[(i)]
			\item $0 < l(x) \leq K$, and 
			\item $c(gx^{-1}, x^{-1}) \leq \frac{\epsilon l(x)}{2}$. 
		\end{enumerate}
	\end{enumerate}
\end{restatable*}

\begin{rem}
	Here, the mysterious looking axiom A5 is encoding the fact that if one had a reasonable action on a graph, then one could approximate geodesics in the graph with uniform quasi-geodesics built from the translates of finitely many paths; it is morally a co-compactness condition expressed solely in terms of the length function.  In fact, we prove that this axiom holds for a fairly wide class of actions in Proposition~\ref{generalspace} and Corollary~\ref{preA5}. 
	
	We also note that if $G$ acts on its Cayley graph then one easily gets that the based length function satisfies these axioms with $K=1$ and $\epsilon = 0$. However, if once considers actions on graphs with more than one orbit of vertices, then one quickly discovers that the correct condition is A5(ii) with $\epsilon \neq 0$. Moreover, scaling the graph by a constant clearly changes the value of $K$. For these reasons, to allow these kinds of deformations, we consider these axioms for more general $K$ and $\epsilon$. 
\end{rem}

It turns out that this is indeed sufficient to prove the following: 

\begin{restatable*}{thm}{graphlike}
	Let $l: G \to \R_{\geq 0}$ be a graph-like length function on a group $G$. Then $l$ is bi-Lipschitz equivalent to some (and hence to all) based length function $l_p$ arising from a locally compact, co-compact, metric $G$-graph and with $\text{Stab}(p) = 1$.
\end{restatable*}

%

Note that in view of Theorem~\ref{badChiswell}, since any finitely generated group admits a hyperbolic length function, the extra axioms are clearly necessary.

\begin{rem}
	We should note that another length function one can extract from an action is the \textit{translation length function}, which has the advantage of not relying on a basepoint. This is the point of view of \cite{Culler1987}. An important result here, building on the work of \cite{Culler1987}, is that of \cite{Parry1991} which states that a translation length function (which is $0$-hyperbolic) always arises from an action on a tree. However, this builds crucially on Chiswell's Theorem~\ref{chiswell} so it seems reasonable to start with Lyndon length functions. 
\end{rem}


	\section{Preliminaries}

	We begin with some preliminary definitions and notation. Let $G$ be a group.
	
%
%
	
	\begin{defn}
		Given a metric, $d:G\times G \to \R_{\geq 0}$, on a group $G$ we say that $d$ is \emph{right-invariant} if $d(g_1 h, g_2 h) = d(g_1, g_2)$ for all $g_1, g_2, h \in G$. 
	\end{defn}
%
%
%
%


	\begin{defn}\label{hyplength}
		A map $l:G\to \R_{\geq 0}$ which satisfies the following axioms is called a \emph{length function}: 
		
			\begin{enumerate}[{A}1:]
			\item $l(g)=0$ if and only if $g=1$
			\item $l(g^{-1}) = l(g)$
			\item $c(g,h) \geqq 0$
		\end{enumerate}
		where
		\begin{align*}
			c(g,h) := \frac{1}{2}(l(g) + l(h) - l(gh^{-1}))
		\end{align*}
		is the \emph{Gromov product} of $g,h\in G$.

		If, in addition, $l$ satisfies
		\begin{enumerate}
			\item[H$_{\delta}$:] $c(g_1,g_2)\geq m$, $c(g_2,g_3)\geq m$ implies that $c(g_1, g_3)\geq m-\delta$ \hspace{35pt}
		\end{enumerate}
		for some $\delta \geq 0$, we say it is a \emph{$\delta$-hyperbolic} length function. The condition $H_{\d}$ is referred to as \emph{$\delta$-hyperbolicity}.
	\end{defn}

\begin{rem}
	Given a length function, it is easy to verify that $d(g,h) := l(gh^{-1})$ is a right-invariant metric on $G$. In particular, A3 is equivalent to the triangle inequality, which can be written as
	\begin{align*}
		l(gh^{-1}) \leqq l(g) + l(h)
	\end{align*}

	Also note that here we write axiom A1: $l(g)=0$ if and only if $g=1$ rather than A1': $l(g)=0$ if $g=1$. This is largely because we end up wanting to characterise those length functions which are bi-Lipschitz equivalent (or quasi-isometric) to those arising from Cayley graphs.  We will sometimes emphasise this by saying that the length function is \emph{free}.
\end{rem}

	
	\begin{defn}
		A \emph{metric graph} is a 1-dimensional CW-complex with a metric structure. A \emph{metric tree} is a metric graph in which any two vertices are connected by exactly one simple path. We always equip metric graphs with the path metric. 
	\end{defn}
	
	\begin{defn}
		By a \emph{metric $G$-graph}, we mean a metric graph $\G$ together with an isometric right action of $G$ on $\G$, sending vertices to vertices and edges to edges. 
		
		Since we think of our graphs as metric spaces, given a point $p$ in $\G$, we may invoke Definition~\ref{basedlength}; $l_p(g) = d_\G(p, p.g)$ is the based length function on $\G$, based at $p$.
	\end{defn}

	\section{Hyperbolicity, Length Functions and Counter-Examples}

		Given a length function, $l$, as in Definition~\ref{hyplength} - that is to say, given a metric on $G$ - one can always construct \textit{some} metric graph on which $G$ acts isometrically and such that $l = l_p$:

	\begin{lem}
		\label{complete}
		Let $l$ be a length function on the group, $G$, as in Definition~\ref{hyplength}. 
		
		Let $\G$ be the complete graph on vertex set $G$, where the length of the edge between $g$ and $h$ is set to $l(gh^{-1}) = l(hg^{-1})$. Then $G$ acts isometrically on $\G$ and $l$ is equal to the based length function on $\G$ - Definition~\ref{basedlength} -  based at the vertex $1$.  
	\end{lem}
	
	However, this is not a very useful object and we will want to insist on some finiteness conditions; namely, co-compactness and (usually) local compactness.

	Since this work arose as an attempt to generalise the celebrated result of Chiswell, Theorem~\ref{chiswell}, it is a natural way to try to generalise that result by weakening 0-hyperbolicity to $\delta$-hyperbolicity and instead only expecting the action to be on a (hyperbolic) graph. It turns out that this doesn't work and we present two counter-examples, in Proposition~\ref{deltacounterexample} and Proposition~\ref{logs}.

	Before presenting the first example, it is worthwhile observing some examples of hyperbolic length functions which \textit{do} arise as the length function of a co-compact action on a graph. In these examples, we can take an existing length function and deform it slightly, but the following examples show that in doing so one might still end up with a length function arising from an action on a graph. 
	
	\begin{exs} For both of these examples, our group is the infinite cyclic group, $\Z$. 
		\begin{enumerate}[(i)]
			\item Given $0 \leq \epsilon < 1$, define: 
			$$ 
			\begin{array}{rcl}
				l(n) &  = &  \left\{ \begin{array}{cl}  1+\epsilon & \text{if  } n=\pm 1 \\ |n| & \text{otherwise} 	
				\end{array}	 \right.
			\end{array}
			$$
			
			One can verify that this is the length function of the Cayley Graph of $\Z$, with respect to the generating set $\{ 1, 2, 3\}$ where $1$ is given length $1 + \epsilon$, $2$ is given length $2$ and $3$ is given length $3$.

			\item Again,  given $0 \leq \epsilon < 1$, define:
			$$ 
			\begin{array}{rcl}
				l(n) &  = &  \left\{ \begin{array}{cl}  0 & \text{if  } n=0 \\ |n|+\epsilon & \text{otherwise}  \\ 
				\end{array}	 \right.
			\end{array}
			$$
			This is actually $0$-hyperbolic, and arises from a non-minimal action on a tree. More precisely, take a graph with two vertices, $u$ and $v$, and two edges, one of which is a loop of length 1 at $v$ and the other is an edge of length $\epsilon/2$ joining $u$ to $v$. The fundamental group of that graph is $\Z$ and the action on the universal cover gives our length function (with respect to any lift of $u$). 
		\end{enumerate}
	\end{exs}
	
Next we show how to deform the standard length function on $\Z$ so as to end up with something that does not arise from an action. 	

	In order to proceed, we need the following observation:

\begin{lem}
	\label{alphas}
	Let $\Gamma$ be a co-compact, metric $G$-graph and $p \in \Gamma$. Let  $l_p(g) = d_{\Gamma}(p, p.g)$ denote the based length function. Then there exist finitely many positive real numbers, $\alpha_1, \ldots, \alpha_k$ such that, for any $g \in G$, $l_p(g)$ belongs to the submonoid of the (additive) real numbers generated by the $\alpha_i$. 
	
	That is, for every $g$, there exist non-negative integers $n_i$ such that $l_p(g) = \sum n_i \alpha_i$.
\end{lem}
\begin{proof}
	We simply let the $\alpha_i$ be the lengths of the edges in $\Gamma$. Since the action is isometric and there are finitely many edge-orbits, it suffices to take only finitely many of them. 
\end{proof}
	
Now we are ready to show that a $\delta$-hyperbolic length function need not come from an action on a graph.

	\begin{prop}\label{deltacounterexample}
		For any $0 \leq \epsilon < 1$, the function $l_{\epsilon}: \Z\to \R_{\geq 0}$ defined by $l_{\epsilon}(n) = |n| + \epsilon^{|n|}$, for $n \neq 0$ and $l(0)=0$ is a hyperbolic length function. 
		
		For $\epsilon = 1/2$, this cannot be equal to any based length function arising from a co-compact, isometric action of $\Z$ on a metric graph.
	\end{prop}
	\begin{proof}
		First we verify axioms A1 to A3 and H$_{\delta}$ from Definition~\ref{hyplength}. Note that for $\epsilon=0$, this is just the standard length function of $\Z$ acting on the line (which is $0$-hyperbolic). For each $l_{\epsilon}$ we define $c_{\epsilon}$ to be the corresponding Gromov product. Note that both $l_0$ and $c_0$ take values in $\Z$. 
		
		Observe that A1 and A2 are clear for all $\epsilon$ directly from the definition. To verify A3, notice that 
		$$l_0(n) \leq l_{\epsilon}(n) \leq l_0(n) + \epsilon, $$
		
		and hence that for any $n,m \in \Z$, 
		$$ c_{\epsilon}(n,m) \geq    c_0(n,m) - \epsilon/2.$$  
		
		Therefore, since $c_0$ takes values in $\Z$, and $\epsilon < 1$, the only values for which $ c_{\epsilon}(n,m)$ could be negative would be those where $c_0(n,m)=0$. Since, for positive integers $n,k$ we have that $c_0(n, n+k) = c_0(-n, -n-k) = n$, we see that $c_0(n,m)$ can only be zero if one of $n,m$ is zero or if one is positive and one is negative. We calculate: if $n,m$ are positive then
		$$
		c_{\epsilon}(0,n) = c_{\epsilon}(0,-n) = 0
		$$
		and,  
		$$
		c_{\epsilon}(n, -m) = \epsilon^n + \epsilon^m - \epsilon^{n+m} > 0. 
		$$
		This verifies A3. To verify H$_{\delta}$, note that the inequality $l_0(n) \leq l_{\epsilon}(n) \leq l_0(n) + \epsilon$ also gives us that $\epsilon + c_0(n,m) \geq c_{\epsilon}(n,m)$. Hence we get, for all $n,m$, 
		$$
		\epsilon + c_0(n,m) \geq c_{\epsilon}(n,m) \geq c_0(n,m) - \epsilon/2.
		$$
		
		But since $l_0$ is $0$-hyperbolic, this implies that $l_{\epsilon}$ is $\frac{3\e}{2}$-hyperbolic.
		
		\medspace
		
		To see that $l_{1/2}$ cannot arise as the length function coming from a co-compact metric $\Z$-graph, we invoke Lemma~\ref{alphas} and argue by contradiction. That is, suppose that $l_{1/2}$ arises from the action of $\Z$ on a co-compact metric graph, $\Gamma$. Then, by Lemma~\ref{alphas}, we have $\alpha_1, \ldots, \alpha_k$ such that for any $g \in G$, there exist positive integers, $n_1, \ldots, n_k$ with $l_{1/2}(g) = \sum_{i=1}^k n_i \alpha_i$. We now show that this is not possible. 
		
		Without loss of generality, by enlarging the set, we may assume that $\alpha_1=1$. Further, again without loss of generality, we may assume that $\alpha_1, \ldots, \alpha_r$ is a maximal, $\Q$-linearly independent subset of the $\alpha_i$. Thus for any $j > r$, $\alpha_j$ is a $\Q$-linear sum of $\alpha_1, \ldots, \alpha_r$. Fix such an expression for each $j$ (in fact, it is unique) and notice that the denominators in the coefficients of these expressions are bounded. In particular, this means that any expression $\sum_{i=1}^k n_i \alpha_i$, where the $n_i$ are integers can be re-written as an expression $\sum_{i=1}^r q_i \alpha_i$, where the $q_i$ are now rational, but with bounded denominator. In particular, this means that there exists an integer, $M$, such that for any $g \in \Z$, there exists integers $m_i$ such that, 
		$$
		l_{1/2}(g) = \frac{1}{M}\sum_{i=1}^r m_i \alpha_i. 
		$$
		
		However, notice that $l_{1/2}(g)$ are rational for every $g$, and the set $\alpha_1, \ldots, \alpha_r$ are $\Q$-linearly independent. Hence the $\Q$-linear independence forces $m_i=0$ for $i \geq 2$, and therefore, 
		$$
		l_{1/2}(g) = \frac{1}{M} m_1 \alpha_1 = \frac{1}{M}m_1. 
		$$
		
		This is clearly impossible, since the values of $l_{1/2}$ do not belong to the additive cyclic subgroup generated by a rational number. 
	\end{proof}
	
	\begin{rem}
		Note that the same proof shows that $l_{\epsilon}$  cannot be equal to any length function arising from a co-compact, isometric action of $\Z$ on a metric graph for any rational $\epsilon$. 
		
	\end{rem}


%

	
	
	The idea of Proposition~\ref{deltacounterexample} is that we started with a $0$-hyperbolic length function (which is the standard length function of $\Z$ acting on its Cayley graph) and deformed it slightly to obtain a length function that is $\delta$-hyperbolic but is not equal to any based length function coming from a co-compact graph. Naturally, since this is a small deformation we obtain a length function which is bi-Lipschitz equivalent to the original length function. We could also consider quasi-isometry.  
	
%
%
%
	
	\begin{defn}
		We say that two length functions $l_1, l_2$ on a group $G$ are \emph{quasi-isometric} if there exists $A \geq 1, B \geq 0$ such that, $\forall g\in G$, 
		$$
		\frac{1}{A} l_1(g) - B \leq l_2(g) \leq A l_1(g) + B.
		$$
		If, in addition, we can take $B=0$, we say that $l_1, l_2$ are bi-Lipschitz equivalent. 
	\end{defn}

We record a standard consequence of the Svarc-Milnor Lemma (see, for example, \cite{Bridson1999}, I.8.19): 

\begin{lem}\label{anynicegraph}
	Let $X,Y$ be co-compact, locally compact metric $G$-graphs. Then for all points $p\in X$, $q\in Y$ such that $\text{Stab}(p) = \text{Stab}(q) = 1$, the based length functions $l_p$ and $l_q$ are bi-Lipschitz equivalent.
\end{lem}
%
%
%
%
%

	Instead of seeking length functions on $G$ which are equal to the based length function of a suitable $G$-graph, we can instead seek $l:G\to \R_{\geq 0}$ which lies in the \emph{quasi-isometry class} of a suitable $G$-graph, ideally a Cayley graph for $G$. Our aim is then to produce axioms for a length function that make it quasi-isometric, or bi-Lipschitz equivalent to a based length function on a Cayley graph. 
	
	Even here it turns out that hyperbolicity is not sufficient.

\begin{prop}
	\label{logs}
	Let $G$ be a finitely generated group and let $|.|:G \to \R$ be the word metric with respect to some finite generating set. Define a function, $l:G \to \R$ by $l(g):=\log(|g|+1)$. Then this is a $\delta$-hyperbolic length function, for a uniform $\delta = \frac{1}{2} \log 32$. 
	
	When $G=\Z$ then $l$ is not quasi-isometric (and hence not bi-Lipschitz equivalent) to \emph{any} based length function on a geodesic and proper hyperbolic space - for an isometric action of $\Z$.
\end{prop}
\begin{proof} First we verify the axioms from Definition~\ref{hyplength}. 
	We immediately see that $l$ satisfies axioms A1 and A2. To see that A3 holds we observe that for all $g,h\in G$, $|g|+|h| \geq |gh^{-1}|$. Thus,
	\begin{align*}
		\log(|g| + 1) + \log(|h|=1) &= \log((|g|+1)(|h|+1))\\
		&= \log(|g||h| + |g| + |h| +1)\\
		& = \log(|g| + |h| +1)\\
		&\geq \log(|gh^{-1}| +1)\\
		\Rightarrow c(g,h) = \frac{1}{2} (l(g) + l(h) - l(gh^{-1})) &\geq 0
	\end{align*}
	Thus $l$ is a length function. 
	
	\medspace
	
	To see that the length function is $\delta$-hyperbolic, consider the function, 
	$$
	d(g,h) := e^{2c(g,h)} = \frac{(|g|+1)(|h|+1)}{|gh^{-1}|+1}, g,h \in G.
	$$
	It will be sufficient to show that there exists a $\delta \geq 0$ such that for any three group elements, $g,h,k$, and any $R \geq 0$,
	$$
	d(g,h) \geq e^{2R} \text{ and } d(h,k) \geq e^{2R} \implies d(g,k) \geq e^{2(R-\delta)}.
	$$
	
	To do this, first observe the following two inequalities:
	
	\begin{eqnarray}
		\label{one} |g| \geq 2|h| \implies 2(|h|+1) \geq d(g,h) \\
		\label{two} d(g,h) \geq \min \{ \frac{|g|+1}{2},  \frac{|h|+1}{2} \}.
	\end{eqnarray}
	
	To see that \ref{one} is true, simply observe that if $|g| \geq 2|h|$ then, 
	$$
	|gh^{-1}| + 1 \geq |g| - |h| + 1 \geq \frac{|g|}{2}+1 \geq \frac{|g|+1}{2},
	$$
	from which it follows that $2(|h|+1) \geq d(g,h)$.  	
	
	To see that \ref{two} is true, observe that if $|g| \geq |h|$ then, $|gh^{-1}|+1 \leq |g| + |h| + 1 \leq 2(|g| + 1)$, from which the desired inequality follows.
	
	To verify that our length function is hyperbolic, let us suppose that we have a triple of group elements, $g,h,k$, and a real number, $R \geq 0$ such that $d(g,h) \geq e^{2R} \text{ and } d(h,k) \geq e^{2R}$. 
	
	Our aim is to find a (uniform) $\delta >0$ such that $d(g,k) \geq e^{2(R-\delta)}$.

	We set $\Lambda = \max\{ |g|, |h|, |k|\}$ and $\lambda = \min\{ |g|, |h|, |k|\}$ the argument breaks into two cases now, depending on whether $\Lambda \geq 4 \lambda$, or $\Lambda < 4 \lambda$.
	
	\noindent
	case(i): $\Lambda \geq 4 \lambda$: 
	
	Without loss of generality, we will assume that $|g| \geq |k|$. In particular this implies, from Equation~\ref{two}, that $d(g,k) \geq \frac{|k|+1}{2}$. 

	Suppose first that $|h| \geq 2|k|$. Then from Equation~\ref{one}, $2(|k|+1) \geq d(h,k)$. Therefore, 
	$$
	d(g,k) \geq \frac{|k|+1}{2} \geq \frac{d(h,k)}{4} \geq \frac{e^{2R}}{4}, 
	$$
	as required (with $\delta=\log(2)$). (We haven't used the fact that $\Lambda \geq 4 \lambda$ yet).

%

	
	If instead we have that, $|h| < 2|k|$ then we must get that $|g| > 2|h|$, since $\Lambda \geq 4 \lambda$. 

	Hence equations \ref{one}, \ref{two} give us that
	$$
	d(g,k) \geq \frac{|k|+1}{2} > \frac{|h|+1}{4} \geq \frac{d(g,h)}{8} \geq \frac{e^{2R}}{8}, 
	$$
	as required (here with $\delta = \frac{1}{2} \log(8)$). 
	
	\noindent
	case(ii): $\Lambda < 4 \lambda$.
	
	Here, we invoke the triangle inequality to get that:
	$$
	|gk^{-1}|+1 \leq 2\max\{ |gh^{-1}|+1, |hk^{-1}|+1\}.
	$$
	Without loss of generality, we assume that $|gh^{-1}| \geq |hk^{-1}|$. Then,
	$$
	d(g,k) \geq \frac{(\lambda+1)^2}{2(|gh^{-1}|+1)} > \frac{(\Lambda+1)^2}{32(|gh^{-1}|+1)} \geq \frac{d(g,h)}{32} \geq \frac{e^{2R}}{32}.
	$$
	
	This completes the proof that our length function is $\delta$-hyperbolic (with $\delta = \frac{1}{2} \log(32)$ as the final and maximal estimate). 
	
	\medspace
	
	To finish, note that if $\Z$ were to act isometrically on a locally compact hyperbolic space, $X$, with based length function $l_p$, then either $l_p$ would have to be bounded, or quasi-isometric to a linear function. Since $\log(|n| +1)$ is neither, it is not quasi-isometric to such an $l_p$. 

\end{proof}

%

\badChiswell

\begin{proof}
	This is simply the content of Propositions~\ref{deltacounterexample} and \ref{logs}. 
\end{proof}

\section{Axioms for graph-like length functions}

We finally turn to positive results and produce an set of axioms that do result in length functions which are bi-Lipschitz equivalent to the based length function on a (or any) Cayley graph. 

Before introducing our axioms, we would like to demonstrate that they are \textit{reasonable}, to the extent that they arise naturally from group actions on fairly general yet well behaved spaces. So we consider the following, noting that the hypotheses on $X$ are satisfied by a locally finite metric graph equipped with the path metric and a co-compact group action.  

\begin{prop}
	\label{generalspace}
	Let $X$ be a geodesic metric space with a given basepoint, $p$. Suppose a group, $G$, acts on $X$ isometrically, and co-boundedly. Then there exist constants, $K>0$ and $0 <\epsilon_0 \leq 1$ such that for any $g \in G$ with $d(p, pg) \geq K$, there exists an $x \in G$ such that: 
	\begin{itemize}
		\item $0 < d(p, px) \leq K$ and, 
		\item $\epsilon_0 d(p, px) + d(px,pg) \leq d(p, pg)$
	\end{itemize} 
\end{prop}
\begin{proof}
	Recall that, 
	\begin{itemize}
		\item $X$ geodesic means that for any two points in $X$ there exists an isometry from a closed real interval to $X$ where the images of the endpoints are our given two points of $X$.
\item The action is co-bounded means there is a closed ball whose $G$ translates cover $X$. 
	\end{itemize}
	
	Since the action is co-bounded, there exists a closed ball centered at $p$, of radius $K/3$ say, whose $G$ translates cover $X$. Set $B=B_{K/3}(p)$ to be this ball. 
	
	Given $g \in G$ with $d(p,pg) \geq K$, let $q \in X$ be the point on a geodesic from $p$ to $pg$ such that $d(p,q) = K/2$. Since $q$ is on a geodesic we also have $d(p,pg) = d(p,q) + d(q, pg)$. 
	
	Now, since the translates of $B$ cover $X$, there exists some $x \in G$ such that $q \in Bx$. This implies that $d(q, px) \leq K/3$. 
	
	First note that $d(p,px) >0$ since, 
	$$
	d(p,px) \geq d(p,q) - d(px,q) \geq K/2 - K/3 = K/6 > 0. 
	$$

	Next note that, 
$$
d(p, px) \leq d(p, q) + d(q, px) \leq K/2 + K/3 = 5K/6.
$$

and also, 

$$
d(px, pg) \leq d(px, q) + d(q, pg) \leq K/3 + (d(p,pg) - K/2) = d(p,pg) - K/6.
$$

Putting these together we get that,  
$$
\frac{1}{5}d(p,px) +   d(px, pg) \leq K/6 + ( d(p,pg) - K/6) = d(p,pg).
$$
Hence we are done, with $\epsilon_0 = 1/5$. 
\end{proof}

\begin{cor}
	\label{preA5}
	With the same hypotheses as above, set: 
	\begin{itemize}
		\item $l_p(g) = d(p,pg)$ and, 
		\item $c_p(g,h) = \frac{1}{2}(l_p(g) + l_p(h) - l_p(gh^{-1}))$. 
	\end{itemize}
	Then there exists $0 \leq \epsilon < 1$ and $K>0$ such that, for any $g \in G$, if $l_p(g) > K$ then there exists an $x \in G$ with: 
	\begin{enumerate}[(i)]
		\item $0 < l_p(x) \leq K$, and 
		\item $c_p(gx^{-1}, x^{-1}) \leq \frac{\epsilon l_p(x)}{2}$. 
	\end{enumerate}
\end{cor}
\begin{proof}
	Just set $\epsilon = 1 - \epsilon_0$ from Proposition~\ref{generalspace} since 
	$$
	\begin{array}{rclc}
		c_p(gx^{-1}, x^{-1}) & \leq &  \frac{\epsilon l_p(x)}{2} & \iff \\ \\
		l_p(gx^{-1}) + l_p(x) - l_p(g) & \leq & \epsilon l_p(x) & \iff \\ \\
		(1-\epsilon) l_p(x) + l_p(gx^{-1}) &  \leq & l_p(g) \\ 
	\end{array}
	$$
	and the last line is equivalent to the conclusion of Proposition~\ref{generalspace} (where we have also used the fact that $l_p(w) = l_p(w^{-1})$ for all $w$ which is just a consequence of the symmetry of the metric). 
\end{proof}

The idea is that a fairly general class of spaces and actions satisfy the equation given by Corollary~\ref{preA5} and hence we will add this as a axiom for our length functions. Therefore, we propose the following.

%
%

\axioms

\begin{rem}
	We note that A4 is really a statement about the action being properly discontinuous, especially in view of Proposition~\ref{sequence}, which says that in the presence of A5, A4 is equivalent to the statement that $B_K$ is finite. 
	
	In view of Proposition~\ref{generalspace} and Corollary~\ref{preA5} one should view A5 as a co-compactness condition; the challenge here was writing an axiom down which could be stated purely in terms of the length function.

	As noted in the introduction, for a standard Cayley graph, its based length function will satisfy these axioms with $K=1$ and $\epsilon = 0$. The A5 condition with $\epsilon=0$ is effectively saying that for every $g \in G$, there is an $x$ of length at most $K$, such that $xp$ lies on a geodesic from $p$ to $pg$.

	For an example of a group acting on a graph where $\epsilon \neq 0$, consider the free group of rank 2, $F_2$, realised as the fundamental group of a graph with two vertices, $u,v$, and three edges: an edge-loop at $u$, $E_u$, an edge-loop at $v$, $E_v$, and an edge from $u$ to $v$, $E_{uv}$. The action of $F_2$ on the universal cover, $T$, of this graph will induced a based length function which is graph-like, but not with $\epsilon =0$. 
	
	Namely, take a lift of, $\overline{u}$ of $u$, as the basepoint of $T$ and consider the orbit of $\overline{u}$ under the group elements corresponding to elements of the fundamental group of the form $g_n=E_{uv} E_v^n {E_{uv}}^{-1}$, for $n \in \Z$. Then  the geodesic from $\overline{u}$ to $\overline{u} g_n$ only meets the orbit of $\overline{u}$ at its endpoints. Hence this cannot satisfy the A5 condition with $\epsilon =0$ for any $K$.   
	
	In fact, it is straightforward to see that any free $F_n$ action on a metric tree - that is, any point in Culler Vogtmann space - satisfies the axioms above, with $K=1$ but not necessarily with $\epsilon=0$.  
\end{rem}

%




Let us start with the following preparatory results:

\begin{lem}\label{discreteness}
	A length function satisfying A4 is discrete.
\end{lem}
\begin{proof}
	Recall that we say a length function $l:G\to \R_{\geq 0}$ is discrete if there exists $\mu > 0$ such that, for all non-trivial $g\in G$, $l(g)\geq \mu$.
	
	If $G = 1$, then $l$ is immediately discrete. Otherwise, take $\mu = \min\{R>0\mid B_R\neq 1\}$. Since A4 holds, this this minimum will be realised by some $R >0$.
\end{proof}

\begin{lem}
	\label{lambdashort}
	Given $l$ satisfying A5, set $\lambda = \frac{1}{1-\epsilon}$. Then for the $g,x$ listed in A5, we have that: 
	$$
	l(gx^{-1}) \leq l(g) - \frac{1}{\lambda}l(x). 
	$$
\end{lem}
\begin{proof}
	By A5,
	\begin{align*}
		c(gx^{-1}, x^{-1}) &\leq \frac{\epsilon l(x)}{2}\\
		\Rightarrow \frac{1}{2}(l(g) + l(x) - l(gx^{-1})) &\leq \frac{\epsilon l(x)}{2}\\
		\Rightarrow l(g) + l(x) - l(gx^{-1}) &\leq \epsilon l(x)\\
		\Rightarrow l(gx^{-1}) &\leq l(g) + (1-\epsilon) l(x)\\
		\Rightarrow l(gx^{-1}) &\leq l(g) - \frac{1}{\lambda}l(x)
	\end{align*}
	
\end{proof}

\begin{lem}\label{generatingball}
	The ball $B_K = \{g\in G \mid l(g) \leq K\}$ is a generating set for $G$. In particular, by A4, $G$ is finitely generated.
\end{lem}
\begin{proof}
	We will show, by induction on $n$, that $\langle B_K \rangle$ contains all group elements $g$ with 
	$$
	l(g)\leq K + \frac{n \mu}{\lambda}
	$$
	(taking $\lambda$ from Lemma~\ref{lambdashort} and $\mu$ from Lemma~\ref{discreteness}) and hence contains all of $G$.
	
	Firstly, take $g\in G$ such that $l(g) \leq K$. Then $g\in B_K \in \langle B_K \rangle$, and we are done.
	
	Now assume that, for all $g\in G$ satisfying $l(g)\leq K + \frac{(n - 1)\mu}{\lambda}$, $g$ lies in $\langle B_K \rangle$.
	
	Take $g$ such that $l(g)\leq K + \frac{n \mu}{\lambda}$. Then, by Lemma~\ref{lambdashort}, there exists $x\in B_k$ such that
	\begin{align*}
		l(gx^{-1}) &\leq l(g) - \frac{1}{\lambda} l(x)\\
		&\leq K + \frac{n \mu}{\lambda} - \frac{\mu}{\lambda} &&\text{(by properties of $g$ and Lemma~\ref{discreteness})}\\
		& = K + \frac{(n - 1)\mu}{\lambda}
	\end{align*}
	Thus $gx^{-1}\in \langle B_K \rangle$, and since $x\in B_K$, this means that $g = gx^{-1}x\in  \langle B_K \rangle$.	
\end{proof}

\begin{prop}\label{sequence}
	Let $l: G \to \R_{\geq 0}$ satisfy A1, A2, A3 and A5. Let $K, \epsilon$ be as in A5, let $\lambda = \frac{1}{1 - \epsilon}$, and suppose that the ball $B_K = \{g\in G \mid l(g) \leq K\}$ is finite. Then,
	\begin{enumerate}[(a)]
		\item For any nontrivial $g \in G$, there exists a finite sequence, $x_0, \ldots, x_k$ such that: 
		\begin{enumerate}[(i)]
			\item Each $0<l(x_i) \leq K$ (i.e. each $x_i\in B_K\setminus \{1\}$),
			\item $l(g {x_0}^{-1} {x_1}^{-1} \ldots {x_k}^{-1}) = 0$, and 
			\item $\frac{1}{\lambda} \sum_{i=0}^k l(x_i) \leq l(g) \leq \sum_{i=0}^k l(x_i)$. 
		\end{enumerate}
	
		\item Axiom A4 holds - that is to say, the ball $B_R$ is finite for all $R\geq 0$.
	\end{enumerate}

\end{prop}

\begin{proof}
	Part \emph{(a)} is clearly true if $l(g) \leq K$, since we can just take $g=x_0$. To prove it in general, we use the discreteness of the length function to argue by induction. More precisely, we let $P_n$ be the statement that \emph{(a)} holds for all $g$ with $l(g) \leq K + \frac{n\mu}{\lambda}$. Thus our initial observation is that $P_0$ holds. We also observe that, since $B_K$ is finite, there exists a minimum length, $\mu >0$, for elements in $B_K$.
	
	Suppose then that $P_{n-1}$ holds and consider a $g \in G$ with $l(g) \leq  K + \frac{n\mu}{ \lambda}$. If $l(g) \leq K$ then we are done, as above. Otherwise, by Lemma~\ref{lambdashort} and the existence of $\mu$, 
	there exists an $x \in G$ with $0<l(x) \leq K$ and 
	
	\begin{equation}
		\label{inequality}
		l(gx^{-1}) \leq l(g) - \frac{1}{\lambda}l(x) \leq K + \frac{(n-1)\mu}{\lambda}. 	
	\end{equation}

	Now by the induction hypothesis applied to $g_0 = gx^{-1}$ we can find $x_1, \ldots, x_k \in G$ such that 
	
	\begin{enumerate}[(i)]
		\item Each $0< l(x_i) \leq K$ 
		\item $l(g_0 {x_1}^{-1} {x_2}^{-1} \ldots {x_k}^{-1}) = 0$, and 
		\item $  \frac{1}{\lambda} \sum_{i=1}^k l(x_i) \leq l(g_0) \leq \sum_{i=1}^k l(x_i)$. 
	\end{enumerate}
	
	Now set $x=x_0$. Then, by Equation~\ref{inequality}, we have that, 
	
	$$
	\frac{1}{\lambda} \sum_{i=0}^k l(x_i) = \frac{1}{\lambda}l(x) + \frac{1}{\lambda} \sum_{i=1}^k l(x_i) \leq \frac{1}{\lambda} l(x) + l(g_0) \leq l(g).  
	$$
	
	Moreover, by A3, 
	$$
	l(g) \leq l(x) + l(g_0) \leq  \sum_{i=0}^k l(x_i). 
	$$
	
	Hence, by induction, \emph{(a)} result holds for all $g$.

	We now prove \emph{(b)}. Let $M\geq 0$, and let $S_M$ denote the following set of finite sequences:
	$$
	S_M = \biggl\{x_0, \ldots x_k \in B_K\setminus \{1\} \mid  \sum_{i=1}^k l(x_i)\leq M \biggr\}
	$$
	Let $R\geq 0$, and let $g\in B_R$. By \emph{(a)}, there exists a sequence $x_0, \ldots x_k\in S_{R\lambda}$ such that $g = x_k\ldots x_0$. Therefore, if we can prove that $S_{R\lambda}$ is finite, we will prove that $B_R$ is finite.
	
	Recall that, for all $x\in B_K$, $l(x)\geq \mu$. Therefore, for all sequences in $S_{R\lambda}$,
	$$
	R\lambda \geq \sum_{i=1}^k l(x_i) \geq k\mu \Rightarrow k\leq \frac{R\lambda}{\mu}
	$$
	Thus $S_{R\lambda}$ is a set of sequences of elements from the finite set $B_K$, and the length $k$ of these sequences has an upper bound. Thus $S_{R\lambda}$ is finite, and we are done.
\end{proof}



%

\graphlike

\begin{proof}
	We take $\G$ to be the Cayley graph on the set $B_K = \{g\in G \mid l(g) \leq K\}$, but instead of assigning every edge length 1, we assign it the length of the corresponding generating element under $l$. That is, the vertex set of $\G$ is $G$, and we join two vertices, $g,h$ by an edge if and only if $gh^{-1} =y \in B_K$; in that case we assign that edge a length of $l(y)$. (Note that $hg^{-1} = y^{-1}$ will also be in $B_K$ in that case and have the same length). $\G$ is then equipped with the path metric.

	We then take the base point $p$ to be the vertex corresponding to the identity. By Lemma~\ref{generatingball}, $B_K$ is a finite generating set for $G$; hence $\G$ is well-defined and the action of $G$ is co-compact.
	
	We can immediately see that $0=l(1) = l_p(1)$, so we shall restrict our attention to nontrivial $g\in G$. For $g\in G$ we write, as always, $l_p(g) = d_{\G}(p, p\cdot g)$ to denote the based length induced by $\G$. The metric on $\G$ is the path metric, and so the distance from $p$ to $pg$ is the infimum of the lengths of all edge paths from $p$ to $p\cdot g$. Thus for all $g\neq 1$,
	\begin{align*}
		l_p(g) &= \inf \big\{ \sum_{i=0}^k l_p(y_i) \mid y_0, \ldots, y_k\in B_K, y_k\ldots y_0 = g \big\}\\
		&= \inf \big\{ \sum_{i=0}^k l(y_i) \mid y_0, \ldots, y_k\in B_K, y_k\ldots y_0 = g \big\}
	\end{align*}
	where the second equality arises from the fact that $l_p(y) = l(y)$ for all $y \in B_K$, as these are the edges of $\G$. 

	By Proposition~\ref{sequence} \emph{(a)}, there exists a sequence $x_0, \ldots, x_k\in B_K$ such that $x_k\ldots x_0 = g$ and $\frac{1}{\lambda} \sum_{i=0}^k l(x_i) \leq l(g)$, where $\lambda = \frac{1}{1-\epsilon}$. Thus $\frac{1}{\lambda} l_p(g) \leq l(g)$.
	
	Conversely, by inductively applying A3, the triangle inequality, $l(g)\leq \sum_{i=0}^k l(y_i) = \sum_{i=0}^k l_p(y_i) $ for all sequences $y_0, \ldots, y_k \in B_K$ with $y_k\ldots y_0 = g$. Thus $l(g) \leq l_p(g)$.
	
	We have $\frac{1}{\lambda} l_p(g) \leq l(g) \leq $ $l_p(g)$, hence $l_p$  is bi-Lipschitz equivalent to $l$ with bi-Lipschitz constant $\lambda = \frac{1}{1-\epsilon}$.

	Therefore, by Corollary~\ref{anynicegraph}, $l$ lies in the bi-Lipschitz equivalence class of all based length functions arising from free, locally compact, co-compact metric $G$-graphs.
\end{proof}

%
%

\begin{rem}
	A hyperbolic graph-like length function is a length function that satisfies the axioms from Definition~\ref{axioms} as well as the $H_{\delta}$ axiom from Definition~\ref{hyplength} for some $\delta > 0$. 
\end{rem}


\newcommand{\etalchar}[1]{$^{#1}$}

\end{document}